\documentclass[11pt,reqno]{amsart}
\usepackage[a4paper]{geometry}

\usepackage{amssymb,latexsym}
\usepackage{graphicx}
\usepackage{url}		
\usepackage{todonotes}
\usepackage{hyperref}
\usepackage{multicol}
\usepackage{amsmath}
\usepackage{mathtools}
\usepackage{rotating}
\usepackage{amsmath}
\usepackage{enumitem}
\usepackage{textcase}
\usepackage{mathdots}
\usepackage{nicematrix}
\usepackage{booktabs}
\usepackage{hyperref}

\usepackage[T2A]{fontenc}
\usepackage{soul}
\setlist{itemsep=8pt}
\usepackage{empheq}
\usepackage{tikz}
\usetikzlibrary{matrix,backgrounds}
\usetikzlibrary{decorations.pathreplacing,calligraphy}
\newlength{\mydrawlinewidth}
\NewDocumentCommand{\markA}{ O{red} m m }{%
    \draw[rounded corners=3mm,opacity=.5, line width=2pt,#1,fill=#1!20] 
    ([xshift=5pt,yshift=-5pt]yannis-#2.north west) 
    rectangle 
    ([xshift=-5pt,yshift=5pt]yannis-#3.south east);
}
\NewDocumentCommand{\markB}{ O{red} m m }{%
    \draw[rounded corners=3mm,opacity=.5,line width=1pt,#1,fill=#1!50]
       ([yshift=2pt]yannis-#2.north) 
    -- ([xshift=2pt]yannis-#3.east)
    -- ([yshift=-2pt]yannis-#3.south)
    -- ([xshift=-2pt]yannis-#2.west)
    -- cycle;
}

\NewDocumentCommand{\markC}{ O{red} m m }{%
    \draw[rounded corners=3mm,opacity=.5,line width=1pt,#1,fill=#1!50]
       ([yshift=2pt]yannis-#3.north) 
    -- ([xshift=2pt]yannis-#2.west)
    -- ([yshift=-2pt]yannis-#2.south)
    -- ([xshift=-2pt]yannis-#3.east)
    -- cycle;
}

\calclayout

\makeatletter
\newcommand{\monthyear}[1]{%
  \def\@monthyear{\uppercase{#1}}}
\newcommand{\volnumber}[1]{%
  \def\@volnumber{\uppercase{#1}}}
\AtBeginDocument{%
\def\ps@plain{\ps@empty
  \def\@oddfoot{\@monthyear \hfil \thepage}%
  \def\@evenfoot{\thepage \hfil \@volnumber}}
\def\ps@firstpage{\ps@plain}
\def\ps@headings{\ps@empty
  \def\@evenhead{%
    \setTrue{runhead}%
    \def\thanks{\protect\thanks@warning}%
    \uppercase{AOM}\hfil}%
  \def\@oddhead{%
    \setTrue{runhead}%
    \def\thanks{\protect\thanks@warning}%
    \hfill\MakeTextUppercase{\othertitle}}%
  \let\@mkboth\markboth
  \def\@evenfoot{%
    \thepage \hfil \@volnumber}%
  \def\@oddfoot{%
    \@monthyear \hfil \thepage}%
  }%
\footskip=25pt
\pagestyle{headings}%
}
\makeatother

\theoremstyle{plain}
\numberwithin{equation}{section}
\newtheorem{thm}{Theorem}[section]
\newtheorem{theorem}[thm]{Theorem}
\newtheorem{lemma}[thm]{Lemma}

\newtheorem{definition}[thm]{Definition}
\newtheorem{proposition}[thm]{Proposition}
\newtheorem{formula}[thm]{Formula}
\newtheorem{notation}[thm]{Notation}
\newtheorem{corollary}[thm]{Corollary}
\newtheorem{remark}[thm]{Remark}

\begin{document}
\monthyear{July 2026}
\volnumber{Volume, Number}
\setcounter{page}{1}

\title{Cassini's identity for $k$-bonacci Numbers}
\def\othertitle{Cassini $k$-bonacci}
\author{Harold R. Parks}
\address{Oregon State University, Corvallis, Oregon, USA.}
\email{hal.parks@oregonstate.edu}
\author{Dean. C. Wills}
\address{Thayer School of Engineering at Dartmouth College, Hanover, NH, USA}
\email{dean.c.wills.th@dartmouth.edu}

\begin{abstract}
Efforts have been made to extend Cassini's identity (also known as Simson's identity) to the $k$-step or $k$-bonacci numbers for decades. These efforts have lacked both completeness of result and simplicity of proof, and this question remains open and relevant. In this note, we offer a definitive solution as well as the generalization of both Catalan's and Vajda's identities. 
\end{abstract}
\maketitle

\vspace{10pt}


%
%

\section{Introduction}

Our purpose in this article is to generalize the identities of
Cassini, Catalan, and Vajda from their original setting
in the Fibonacci numbers to 
the larger realm of the $k$-bonacci numbers.
Before we begin that work, we will briefly review for the 
reader what the identities of
Cassini, Catalan, and Vajda tell us.

Cassini's identity for the Fibonacci numbers $F_n$ is
\begin{equation} \label{eq:cassini}
F_{n}^{2}-F_{n-1}F_{n+1}=(-1)^{n-1}.
\end{equation}
This identity was first published by  
Jean-Dominque Cassini\footnote{Cassini's given name was 
Giovanni Domenico, but he changed it when he moved to France. 
Jean-Dominque Cassini, his son, Jacques, 
his grandson, C\'esar-Fran\c{c}ois, and his
great-grandson, Jean-Domenique, 
all noted astronomers, are often referred to as 
Cassini I, II, III, and IV.} 
in 1680 (see \cite{bhlpage3208626}).
The identity was independently rediscovered and published by Robert Simson in 1753 (see \cite{simson1753}), so \eqref{eq:cassini} is also known as Simson's identity.

In 1879, 
Eug\`ene Catalan generalized Cassini's identity to
\begin{equation} \label{eq:catalan}
F_{n}^{2}-F_{n-r}F_{n+r}=(-1)^{n-r} F_{r}^2,
\end{equation}
though \eqref{eq:catalan}  was not published until 1886
when it appeared in \cite{catalan}. 

The twentieth century saw a further generalization 
of Cassini's identity to
\begin{equation} \label{eq:vajda}
F_{n}F_{n+p+q}-F_{n+p}F_{n+q}=(-1)^{n-1}F_{p}F_{q}.
\end{equation}
While \eqref{eq:vajda} was published  in 1901
by Alberto Tagiuri
in  \cite{tagiuri1901} and by 
Arthur Danese in 1960 
in the {\it American Mathematical Monthly}
as an elementary problem  
(see \cite{danese1960}), its inclusion 
as  equation (20a) on
page 28 of
Steven Vajda's 1989 book \cite{vajda1989}, 
has led to \eqref{eq:vajda} 
being known as Vajda's identity.\footnote{Vajda does give  credit to
Danese.} 
To see that \eqref{eq:vajda}
is a generalization of \eqref{eq:catalan}, we set $p=r$ and
$q=-r$ in \eqref{eq:vajda}, and then we use 
the fact that $F_{-r} = (-1)^{1-r} F_r$.

\medskip
For our generalizations of \eqref{eq:cassini}, 
\eqref{eq:catalan}, and \eqref{eq:vajda},
we simply replace each Fibonacci number 
$F_\ell$ on the left-hand side
by the corresponding $k$-bonacci number
$F^{(k)}_\ell$. The question then, and the problem we address, is  

\begin{center}
    ``What goes on the right-hand side of each equation?''
\end{center}

For $k>2$ the correct right-hand side cannot be obtained by the  simple technique of replacing each Fibonacci number
$F_\ell$ by the corresponding 
$k$-bonacci number $F^{(k)}_\ell$---the result
of that process would be false.
Instead, we shall obtain the correct right-hand sides by using
combinations of generalized Vandermonde determinants 
based on the roots
of the characteristic polynomial for the $k$-bonacci recursion.

\medskip
It is straightforward to prove Cassini's identity \eqref{eq:cassini} 
using Binet's formula for the Fibonnaci numbers, so
motivation for our approach is provided by the existence of a similar Binet-type formula for the \mbox{$k$-bonacci} numbers
(see \cite{kilicc2006generalized}
or \cite{generalizedBinet2023Parks}).
That Binet formula for the \mbox{$k$-bonacci} numbers
involves a ratio of $k\times k$ determinants that
can reasonably be called generalized Vandermonde determinants. Our
main tool for obtaining our 
generalizations of \eqref{eq:cassini}, 
\eqref{eq:catalan}, and \eqref{eq:vajda}
will also be  ratios of such determinants.
So inspired by E. R. Heineman's notation in
\cite{heineman1929generalized}, we find it helpful
to introduce the following:
\begin{notation}\rm
Let $\ell$ be a positive integer and let $\phi_i$ be $\ell$ elements of a commutative ring $\mathbf{R}$. Then we set
\begin{equation}\label{eq:heineman}
\| \ t_1,\ t_2,\ \dots, \ t_{\ell}\ \|
=
\begin{vmatrix}
\phi_1^{t_1} & \phi_1^{t_2} & \cdots & \phi_1^{t_{\ell-1}} & \phi_1^{t_{\ell}}\\[0.5ex]
\phi_2^{t_1} & \phi_2^{t_2} & \cdots & \phi_2^{t_{\ell-1}}& \phi_2^{t_{\ell}}\\[0.5ex]
\phi_3^{t_1} & \phi_3^{t_2} & \cdots & \phi_3^{t_{\ell-1}}& \phi_3^{t_{\ell}}\\
\vdots & \vdots  & \ddots & \vdots & \vdots \\
\phi_\ell^{t_1} & \phi_\ell^{t_2} & \cdots & \phi_\ell^{t_{\ell-1}} & \phi_\ell^{t_{\ell}}
\end{vmatrix}
\end{equation}
where the $t_i$ are integers. The  $\phi_i$ do not
appear on the left-hand side of \eqref{eq:heineman} and so 
must be understood from the context.
\end{notation}

Using the above notation, with $\ell=k$ and the  $\phi_i$ the roots over $\mathbb{C}$ of the characteristic polynomial of the 
$k$-bonacci recursion,
\[
x^k-x^{k-1}-\cdots-1=0
\,,
\] 
the Binet formula of \cite{generalizedBinet2023Parks} for 
the \mbox{$k$-bonacci} numbers is
\begin{equation}\label{eq.binet}
F_n^{(k)}=
\frac{
\|\ 0, \ 1,\ 2,\ \dots\ k-2,\ \ n\ \|  }
{\|\ 0, \ 1,\ 2,\ \dots\ k-2,\ \ k-1\ \|}
\,.
\end{equation}	
Our results here will be expressed in terms of
\begin{equation}\label{eq.D}
F_{m,n}^{(k)}=
\frac{
\|\ 0, \ 1,\ 2,\ \dots\ k-3,\ m,\ n\ \|  }
{\|\ 0, \ 1,\ 2,\ \dots\ k-3,\ k-2,\ k-1\ \|}
\,.
\end{equation}
Now,  
$F^{(k)}_{m,\,n}$ 
would be  difficult to 
compute directly using   \eqref{eq.D}, especially when $k$ is large.
Fortunately there are other ways to evaluate $F^{(k)}_{m,\,n}$. 
Additionally, $F^{(k)}_{m,\,n}$ satisfies the $k$-bonacci recursion in
both $m$ and $n$.
We will also see from Theorem~\ref{th:vajda1} and Theorem~\ref{thm.reduce.to.k} that
\[
F^{(k)}_{m,n}  =   F^{(k)}_{n}F^{(k)}_{m+1} -F^{(k)}_{n+1}F^{(k)}_{m} 
= F^{(k)}_{n-k}F^{(k)}_{m} -F^{(k)}_{n}F^{(k)}_{m-k} \,,
\]
thus expressing $F^{(k)}_{m,\,n}$ in terms of $k$-bonacci numbers.

The generalization of Cassini's identity is the following:
For all $k\geq2$, it holds that
\[
\left(F^{(k)}_{n}\right)^{2}-
F^{(k)}_{n-1}\,F^{(k)}_{n+1}=F_{ n-1,\, n}^{(k)}\,.
\]
When $k=3$, we get the following  nice result:
\[
    \left(F^{(3)}_{n}\right)^{2}-F^{(3)}_{n-1}\,F^{(3)}_{n+1}=F^{(3)}_{1-n} \,,
\]
and when $k=2$, our generalization reduces to Cassini's original identity.

The most general of our results is the generalization
of Vajda's identity given in Theorem~\ref{th:vajda} where
we  show for $p > 0$
\begin{equation} 
F^{(k)}_{n}F^{(k)}_{n+p+q}-F^{(k)}_{n+p}F^{(k)}_{n+q} = \sum_{i=0}^{p-1}F^{(k)}_{n+q+i,\, n+p-1-i} \,,
\label{eq:vaj}
\end{equation}

\medskip
While it is far from apparent, \eqref{eq:vaj} does simplify 
to Vajda's identity when $k=2$. 
Unfortunately, for $k\geq 3$,the right-hand side of \eqref{eq:vaj},
as it stands, is no longer a simple expression like \eqref{eq:vajda}. 
Nonetheless equation \eqref{eq:vaj}
answers to the question, ``What goes on the right-hand side?''
in the context of our functions $F^{(k)}_{m,\,n}$.


%
%





\section{Preliminaries}
The generalization of Vajda's identity is a consequence of an identity that holds for the particular Vandermonde matrices given in the next definition.
\begin{definition}\label{def.vanders}\rm
Let $\phi_i$ be $k$ elements of a commutative ring  
$\mathbf{R}$.
\begin{enumerate}

\item 
Set
\[
V^{(k)}= V^{(k)}|_{\phi_i} =  \|\ 0,\ 1,\ \dots,\ k-1\ \|
\,.
\]
\item
For $u\in\{ 1,2,\dots,k \}$, we let $V^{(k)}_{\widehat u}$ be the 
$(k-1)\times(k-1)$ determinant obtained from $V^{(k)}$
by omitting row $u$ and  column $k$.

\item
For $u,v\in\{ 1,2,\dots,k \}$ with $u < v $, we let 
$V^{(k)}_{\widehat u\, \widehat v}$ be the 
$(k-2)\times(k-2)$ determinant obtained from $V^{(k)}$
by omitting rows $u$ and $v$ and  
columns $k-1$ and $k$.

\end{enumerate}

\end{definition}

\medskip
All three of $V^{(k)}$, $V^{(k)}_{\widehat u}$, 
and $V^{(k)}_{\widehat u\,\widehat v}$ 
in Definition~\ref{def.vanders} are Vandermonde 
determinants, so the value of each is the 
product of the differences, higher row 
number minus lower row number, of the 
entries in the second column of that particular determinant, when the matrix is oriented as in \eqref{eq:heineman}. 
These facts allow us to prove the next lemma.
\begin{lemma}\label{tricky.identity}
$\phi_i$ be $k$ elements of a commutative ring $\mathbf{R}$.
Suppose $1\leq u < v\leq k$. Then we have
\rm\begin{enumerate}
\item \label{tricky:one}

\ $\displaystyle
V^{(k)} = \prod_{1\leq i< j\leq k} (\phi_j-\phi_i)
\,,
$
\item \label{tricky:two}

\ $\displaystyle
V^{(k)}=
V^{(k)}_{\widehat u} \cdot \prod_{1\leq i < u} (\phi_u-\phi_i)
\cdot \prod_{u<j \leq k} (\phi_j-\phi_u)
\,,$
\item \label{tricky:three}
\begin{align*}
& (\phi_v-\phi_u) \cdot V^{(k)}     \\
&= V^{(k)}_{\widehat u\,\widehat v} 
\cdot \prod_{1\leq i < u} (\phi_u-\phi_i) \cdot \prod_{u<j \leq k} (\phi_j-\phi_u)  
\cdot \prod_{1\leq i < v} (\phi_v-\phi_i) \cdot \prod_{v<j \leq k} (\phi_j-\phi_v)
\end{align*}
\item\label{tricky:four}
\begin{equation}\label{eq:tricky:four}
(\phi_v-\phi_u)\, V^{(k)}_{\widehat u}\, V^{(k)}_{\widehat v}
= V^{(k)}\,V^{(k)}_{\widehat u\,\widehat v}
\,.
\end{equation}
\end{enumerate}
\end{lemma}
\begin{proof} 
Part \eqref{tricky:one}.\  This is the standard
result concerning Vandermonde determinants. For completeness,
we  cite pages 12 and 13 of  \cite{franklin}.

\medskip
\noindent
Part \eqref{tricky:two}.\  The terms 
$\prod_{1\leq i < v} (\phi_v-\phi_i)$ and 
$\prod_{u<j \leq k} (\phi_j-\phi_u)$
on the right-hand side of Part~(\ref{tricky:two}) account for all the 
differences from $\prod_{1\leq i< j\leq k} (\phi_j-\phi_i)$ that contain $\phi_u$.

\medskip
\noindent
Part \eqref{tricky:three}.\  As in Part~\eqref{tricky:two}, the  terms 
$\prod_{1\leq i < u} (\phi_u-\phi_i)$ and 
$\prod_{u<j \leq k} (\phi_j-\phi_u)$
on the right-hand side of \eqref{tricky:three} account for all the 
differences that contain $\phi_u$ from $\prod_{1\leq i< j\leq k}(\phi_j-\phi_i)$.
Likewise, the terms the  terms 
$\prod_{1\leq i < v} (\phi_v-\phi_i)$ and 
$\prod_{v<j \leq k} (\phi_j-\phi_v)$
on the right-hand side of \eqref{tricky:three} account for all the 
differences that contain $\phi_v$ from $\prod_{1\leq i< j\leq k}(\phi_j-\phi_i)$. But in that process, the term $(\phi_v-\phi_u)$ is eliminated
twice, so the last factor $(\phi_v-\phi_u)$ ensures that
$(\phi_v-\phi_u)$ is eliminated exactly once.

\medskip
\noindent
Part \eqref{tricky:four}  
follows immediately from Parts \eqref{tricky:one}, \eqref{tricky:two} and \eqref{tricky:three}.
\end{proof}


Next,
we collect some classical results
concerning determinants that will be used later.
One source for details is \cite{aitken}, specifically
Sections 19, 33, 35, and 36. 

Let $M$ be an $n\times n$ matrix, $n\geq 2$, with
entries $m_{i\,j}$.
For each $u$, let $M_{\widehat{u}}$ be the 
$(n-1)\times(n-1)$ matrix obtained by deleting row $u$
and column $n$ from $M$. Of course, the reader is familiar
with the following Laplace expansion: 
\begin{formula} \label{laplace.1}
\begin{equation*} 
\hbox{\rm det}\, M = \sum_{u=1}^n (-1)^{u+n}\, m_{u\,n}\, 
\hbox{\rm det}\, M_{\widehat{u}}
\,.
\end{equation*}
\end{formula}

When the preceding formula is applied to $V^{(k)}_{m}$,
we obtain
\begin{equation}\label{le:laplace}
V^{(k)}_{m} = \sum_{u=1}^k (-1)^{k+u} \, \phi_u^{m} V^{(k)}_{\widehat u}    \,.
\end{equation}

The reader can certainly imagine applying the above Laplace
expansion to to each summand in 
Formula~\ref{laplace.1}. To do so,  one defines
$M_{\widehat{u}\,\widehat{v}}$ to be the 
$(n-2)\times(n-2)$ matrix obtained by deleting rows $u$ and $v$
and columns $n-1$ and  $n$ from $M$. The challenging part
is keeping track of  the signs. 
The resulting expansion---also due to Laplace---can
be written as follows:
\begin{formula} \rm \label{laplace.2}
\begin{equation*} 
\hbox{\rm det}\, M = \sum_{1\leq u< v\leq n} (-1)^{u+v+1}\, 
\begin{vmatrix}
m_{u\, n-1} & m_{u\,n} \\
m_{v\,n-1} & m_{v\, n}
\end{vmatrix} \, 
\hbox{\rm det}\, M_{\widehat{u}\,\widehat{v}}
\,,
\end{equation*}
where we use the vertical bar notation for $2\times2$ determinants,
because it is a bit more efficient, and it is easily recognized.
\end{formula} 


Another classical result is the following:
\begin{formula}[Binet--Cauchy identity]\rm \label{fo:binet-cauchy}
Suppose $A$ is a $m\times n$ matrix and $B$ is an $n\times m$ matrix 
over a commutative ring. For each subset $S \subseteq \{1 \dots n\}$ 
of size $m$, we can form square matrices $ A_S $ whose columns, and 
$ B_S $ whose rows, are indexed by $S$. Then 
\begin{equation*} 
\det{AB}=\sum_{|S|=m} \det A_S\, \det B_S
\end{equation*}
\end{formula}

The Laplace expansion dates back to 1772. It can be 
further generalized beyond the 
form that we have given above. The 
Binet--Cauchy identity was discovered independently
by Binet and Cauchy in 1812.

%
%
\section{The Main Identity}

\begin{definition}\label{def.vanders1}\rm
We introduce some additional notation for generalized Vandermonde 
determinants. Let $\phi_i$ be $k$ elements of a commutative ring  
$\mathbf{R}$, and let $\ell\in \mathbb{Z}$ such that 
$1\leq \ell \leq k$, and let $c_1, \cdots, c_{\ell}\in \mathbb{Z}$. 
We define
\begin{equation}
    V^{(k)}_{c_1,\ \dots \ c_{\ell}}  
    =   
    V^{(k)}_{c_1,\ \dots \ c_{\ell}}|_{\phi_i} 
    = \|\ 0, \ 1,\ 2,\ \dots,\ k-\ell-1,
\ c_1,\ \dots,\ c_{\ell}\ \| \,.
\end{equation}
In this paper, we are most interested in $\ell \leq 2$. When $\ell=0$, we simply have $V^{(k)}$. When  the 
$\phi_i$ are the $k$ roots of the $k$-bonacci 
characteristic polynomial, 
$\ell=1$ leads to $V^{(k)}_{n}=V^{(k)}F^{(k)}_n$, and $\ell=2$ leads 
to $V^{(k)}_{m,\,n}=V^{(k)}F^{(k)}_{m,\,n}$.
\end{definition}

The next theorem gives an identity that
is the key result for our generalization of 
the identities of Cassini, Catalan, and Vajda. 

\begin{theorem}\label{lemma:second_}
Let the $\phi_i$ be $k$ elements of a field $\mathbb{F}$.
For integers $n$, $p$, and $q$, $p \geq 0$, then it holds that 
\begin{equation}\label{eq.new}
V^{(k)}_{n}\,V^{(k)}_{n+p+q} - V^{(k)}_{n+p}\,V^{(k)}_{n+q} =  \sum_{i=0}^{p-1}V^{(k)}\,
V^{(k)}_{n+q+i,\, n+p-1-i} \,.
\end{equation}
\end{theorem}

\begin{remark}\rm\quad
\begin{enumerate}
\item If $q$ and $n+q$ are non-negative, then the requirement that $\mathbb{F}$
be a field can be relaxed to requiring only that
$\mathbb{F}$ be a commutative ring.
\item
To handle the case $p<0$, replace $n$ by $n+p+q$, replace $p$ by $-p$, and replace $q$ by $-q$. The left-hand side of \eqref{eq.new} is unchanged,
while the right-hand side becomes 
\[ 
\sum_{i=0}^{|p|-1}V^{(k)}\, V^{(k)}_{n+p+i,\, n+q-1-i} 
\,.
\]
\end{enumerate}

\end{remark}

\begin{proof} We see that \eqref{eq.new} holds if $p=0$ as both sides evaluate to
zero. Also, both sides of \eqref{eq.new} vanish 
if two or more of the $\phi_i$ are identical, as the matrices are all singular and
their determinants are zero. 
We may therefore assume that $p > 0 $ and that the $\phi_i$ are distinct.

Applying \eqref{laplace.1} to each of $V^{(k)}_{n}$,
$V^{(k)}_{n+p+q}$, $V_{n+p}^{(k)}$, and $V_{n+q}^{(k)}$,
we have
\begin{align}\nonumber
  &   V^{(k)}_{n}V^{(k)}_{n+p+q} - V_{n+p}^{(k)}V_{n+q}^{(k)}     
  \\[1ex]
  \nonumber
        &\quad= \left(\sum_{u=1}^k   (-1)^{k+u}
        \,\phi_u^n \,V^{(k)}_{\widehat{u}}\right)
        \left(\sum_{v=1}^k   (-1)^{k+v}
        \,\phi_v^{n+p+q} \,V^{(k)}_{\widehat{v}}\right) \\
        \nonumber
        &\qquad- \left(\sum_{u=1}^k  (-1)^{k+u}
        \,\phi_u^{n+p}\,V^{(k)}_{\widehat{u}}\right)
        \left(\sum_{v=1}^k
        (-1)^{k+v}\,\phi_v^{n+q}\,V^{(k)}_{\widehat{v}}\right) \\[1ex]
        \nonumber
        &\quad=\ \begin{vmatrix} 
         \left(\sum_{u=1}^k   (-1)^{k+u}
        \,\phi_u^n \,V^{(k)}_{\widehat{u}}\right)
        & \left(\sum_{v=1}^k
        (-1)^{k+v}\,\phi_v^{n+q}\,V^{(k)}_{\widehat{v}}\right)  \\
         \left(\sum_{u=1}^k  (-1)^{k+u}
        \,\phi_u^{n+p}\,V^{(k)}_{\widehat{u}}\right)
        & \left(\sum_{v=1}^k   (-1)^{k+v}
        \,\phi_v^{n+p+q} \,V^{(k)}_{\widehat{v}}\right)
        \end{vmatrix} \\[1ex]
        \label{determinant.1}
        &\quad=\ \det \left[
        \begin{pmatrix} 
        1 & \dots &1 \\
        \phi^{p}_1 & \dots & \phi^{p}_k \\
        \end{pmatrix} 
        \begin{pmatrix} 
        (-1)^{1}V^{(k)}_{\widehat{1}}\phi^{n}_1 & (-1)^{1}V^{(k)}_{\widehat{1}}\phi^{n+q}_1 \\ 
        \vdots & \vdots \\  
        (-1)^{k}V^{(k)}_{\widehat{k}}\phi^{n}_k & (-1)^{k}V^{(k)}_{\widehat{k}}\phi^{n+q}_k \\
        \end{pmatrix}
        \right] 
\,.
\end{align}
By the Binet--Cauchy identity \eqref{fo:binet-cauchy}, the
determinant in  \eqref{determinant.1}  equals  
\begin{align}\nonumber
& \sum_{1\leq u< v\leq k} 
       \ \begin{vmatrix} 
        1 & 1 \\
        \phi^{p}_u & \phi^{p}_v 
        \nonumber
        \end{vmatrix}    
      \ \ \begin{vmatrix} 
        (-1)^{u}V^{(k)}_{\widehat{u}}\phi^{n}_u & (-1)^{u}V^{(k)}_{\widehat{u}}\phi^{n+q}_u \\ 
        (-1)^{v}V^{(k)}_{\widehat{u}}\phi^{n}_v & (-1)^{v}V^{(k)}_{\widehat{v}}\phi^{n+q}_v \\ 
        \end{vmatrix}  \\[1ex]  
        \nonumber
 &\qquad =  \sum_{1\leq u< v\leq k}  (-1)^{u+v}   (\phi_v^{p}-\phi_u^{p}) \,     V^{(k)}_{\widehat u}\, V^{(k)}_{\widehat v}
 \ \begin{vmatrix} 
        \phi^{n}_u & \phi^{n+q}_u \\ 
        \phi^{n}_v & \phi^{n+q}_v \\ 
        \end{vmatrix}  \\[1ex]
\nonumber & \text{(then factor out $\phi_v-\phi_u$ to obtain)} \\
\nonumber
\nonumber
&\qquad=\sum_{1\leq u< v\leq k}  (-1)^{u+v}   (\phi_v-\phi_u)    V^{(k)}_{\widehat u}\, V^{(k)}_{\widehat v} \left(\sum_{i=0}^{p-1}\phi_v^{p-1-i}\,\phi_u^{i}
\right) \, 
 \  \begin{vmatrix} 
        \phi^{n}_u & \phi^{n+q}_u \\ 
        \phi^{n}_v & \phi^{n+q}_v \\ 
        \end{vmatrix}   \\[1ex] 
\label{set.me.up}
&\qquad
=\  \sum_{1\leq u< v\leq k}  (-1)^{u+v}   
\,     V^{(k)} \, V^{(k)}_{\widehat u\,\widehat v}  \left(\sum_{i=0}^{p-1}\phi_v^{p-1-i}\,\phi_u^{i}
\right)
\, \begin{vmatrix} 
        \phi^{n}_u & \phi^{n+q}_u \\ 
        \phi^{n}_v & \phi^{n+q}_v \\ 
        \end{vmatrix},  
\end{align}
where \eqref{set.me.up} follows from \eqref{eq:tricky:four}.
Now we would like to move each term 
$ \phi_v^{p-1-i}\,\phi_u^{i} $ in the summation 
in \eqref{set.me.up} into 
the determinant in \eqref{set.me.up}. To do so, we expand the
determinant, thus giving us a summation with a plus sign
and a summation with a minus sign. When we reverse the
direction  in the summation with the minus sign,
the terms line up to form a new determinant. 
\begin{align}
\label{after.the.switch}
&  \sum_{i=0}^{p-1} V^{(k)} 
\sum_{1\leq u< v\leq k}  (-1)^{u+v+1}\, 
\begin{vmatrix}
        \phi_u^{n+q+i} & \phi_u^{n+p-1-i} \\
    \phi_v^{n+q+i} & \phi_v^{n+p-1-i}
    \end{vmatrix}  V^{(k)}_{\widehat u\,\widehat v}
    \, \\[1ex]
\label{last.step}
&\quad=\  \sum_{i=0}^{p-1}V^{(k)}V^{(k)}_{n+q+i,\, n+p-1-i} \,,
\end{align}
where 
the Laplace expansion \eqref{laplace.2} 
gets us from \eqref{after.the.switch}   to \eqref{last.step}.
\end{proof}

\section{Focusing on k-bonacci}

\begin{lemma}
    If $\phi_i$ are $k$ unique elements of a commutative ring $\mathbf{R}$, whose minimum monic polynomial has coefficients in a subring $\mathbf{S} \subseteq \mathbf{R}$, and let $\ell\in \mathbb{Z}$ such that $1\leq \ell \leq k$, and let $c_1, \cdots, c_{\ell}\in \mathbb{Z}$, then
\begin{equation}\label{eq:Binet}
 \frac{V^{(k)}_{c_1,c_2, \dots,  c_{\ell}}}{V^{(k)}} \in \mathbf{S}
\end{equation}
\end{lemma}
\begin{proof}
Equating any two $\phi_i$ leads to a singular matrix so the 
numerator is 
divisible by the difference-product of the $\phi_i$, 
which is exactly the 
denominator.\footnote{This proof expands an argument 
by Heineman in \cite{heineman1929generalized}}
\end{proof}
\begin{lemma}\label{co:orthogonal}
If $\phi_i$ are the roots of the $k$-bonacci characteristic polynomial,
then $V^{(k)}_{c_1,\ \dots \ c_{\ell}}$ satisfies the
$k$-bonacci recursion in each index, that is,
\begin{equation}
    V^{(k)}_{c_1, \dots , c_{t}, \dots, c_{\ell}} = \sum_{i=1}^{k} 
    V^{(k)}_{c_1, \dots , (c_{t}-i), \dots, c_{\ell}}
\end{equation}
\end{lemma}
\begin{proof}
    This follows from the definitions and the multilinearity of determinants.
\end{proof}

\begin{definition}\rm
    Considering these last two lemmas in the case that the 
$\phi_i$ are the $k$ distinct roots of the characteristic polynomial
for the $k$-bonacci recursion, we are motivated to generalize the 
$k$-bonacci numbers by defining the 
$(\ell \times k)$-bonacci numbers 
\begin{equation} \label{eq:newFibonacci}
F^{(k)}_{c_1, \dots,  c_{\ell}} = \left. \frac{V^{(k)}_{c_1, \dots,  
c_{\ell}}}{V^{(k)}}\right|_{\phi_i} \,,
    \end{equation} 
which are integer valued and satisfy the $k$-bonacci recursion in each 
of their indices.

We note that equations \eqref{eq.binet} and \eqref{eq.D} 
are consistent with
this last definition.    
\end{definition}
\begin{theorem}[Vajda's identity for the $k$-bonacci 
numbers]\label{th:vajda}  For integers $n, \ p$  and $q$, $p\geq 0$,
it holds that
\begin{equation}\label{gen.vajda}
F^{(k)}_{n}F^{(k)}_{n+p+q}-F^{(k)}_{n+p}F^{(k)}_{n+q} = 
\sum_{i=0}^{p-1}F^{(k)}_{n+q+i,\, n+p-1-i} \,, \end{equation}
\end{theorem}
\begin{proof}
Let the  $\phi_i$  be the $k$ distinct roots of the $k$-bonacci 
characteristic polynomial. Then the 
result follows immediately from Theorem~\ref{lemma:second_}
and \eqref{eq:newFibonacci}. 
\end{proof}

\begin{corollary}[Catalan's identity for the $k$-bonacci numbers]
\label{th:catalan}  For any integer $n$, and $r\geq 1$, 
it holds that
\begin{equation}\label{gen.catalan}
\left( F^{(k)}_{n}\right)^2-F^{(k)}_{n-r}F^{(k)}_{n+r} = 
F^{(k)}_{n-1,\, n} + F^{(k)}_{n-2,\, n+1} +\cdots+
F^{(k)}_{n-r,\, n+r-1}
\,.
\end{equation}
\end{corollary}
\begin{proof}
Let $r\geq 1$ be given. When we define $p=r$ and $q=-r$, and apply 
Theorem~\ref{th:vajda}, we see that \eqref{gen.vajda} 
reduces to \eqref{gen.catalan}.
\end{proof}
Taking $r=1$ in the last corollary, we obtain the following:
\begin{corollary}[Cassini's identity for the $k$-bonacci numbers]
\label{th:cassini}
For any integer $n$,  it holds that 
    \begin{equation}\label{gen.cassini}
\left(F^{(k)}_{n}\right)^{2}-
F^{(k)}_{n-1}\,F^{(k)}_{n+1}=F_{ n-1,\, n}^{(k)}\,.
\end{equation}
\end{corollary}

We will need this for the next result.

\begin{lemma} For $k\geq 2$ and $n\in\mathbb{Z}$, it holds that
\begin{equation}\label{further.calc}
2F^{(k)}_{n} = F^{(k)}_{n+1}+F^{(k)}_{n-k}
\end{equation}
\end{lemma}

\begin{proof}
Equation
\eqref{further.calc} is a consequence of  associativity: 
$ F^{(k)}_{n}+F^{(k)}_{n-1} + 
\cdots + F^{(k)}_{n-k-1} + F^{(k)}_{n-k} $ 
equals itself, and \eqref{further.calc} is obtained by 
grouping the terms  in the two ways illustrated below
\begin{center}
\begin{tikzpicture}[decoration=brace]

\draw (0,0) node[xshift=1.7cm] { $F^{(k)}_{n}+F^{(k)}_{n-1} + \cdots + F^{(k)}_{n-k-1} + F^{(k)}_{n-k}$};

\draw [
    thick,
    decoration={
        brace,
        raise=0.5cm
    },
    decorate
] (-1.2,0) -- (3.2,0);

\draw [
    thick,
    decoration={
        brace,
        mirror,
        raise=0.5cm
    },
    decorate
] (0,0) -- (4.7,0) ;
\draw (2.3,-1) node {$F^{(k)}_{n}$};

\draw (1,1) node {$F^{(k)}_{n+1}$};
\end{tikzpicture}
\end{center}
\end{proof}

It will be useful to note that $F^{(k)}_{m,n}$ is anti-symmetric with respect to $m$ and $n$.
\begin{lemma}
\begin{equation}
    F^{(k)}_{m,n}= - F^{(k)}_{n,m}.
\end{equation}
\begin{proof}
    This follows from the definition, as exchanging two 
    columns of the determinant reverses the sign.
\end{proof}
\end{lemma}
The next lemma tells us that 
for moderately sized values of $m$ and $n$, the value of 
$F^{(k)}_{m,\,n}$ can be obtained with little work. 

\begin{lemma}\label{lemma.values}
For $k\geq 3$ and $m,n\in \mathbb{Z}$, it holds that
\begin{subequations}
    \begin{empheq}[left={F^{(k)}_{m,n}=\empheqlbrace\,}]{align}
& -F^{(k)}_{n+1} & m=-1, \label{eq:7.1.1}\\
& 0 & 0 \leq m \leq k-3,  \label{eq:7.1.2} \\
& F^{(k)}_{n} & m=k-2 ,  \label{eq:7.1.3} \\
& F^{(k)}_{n} - F^{(k)}_{n+1} & m=k-1, \label{eq:7.1.4} \\
&  2^{a-k}F^{(k)}_{n-k}& k \leq m \leq 2k-2. \label{eq:7.1.5}     \\
& -F^{(k)}_{n} & m=-3 ,  \label{eq:7.1.6} 
    \end{empheq}
\end{subequations}

\end{lemma}
\begin{proof}
\eqref{eq:7.1.1}\ \ For $m=-1$, we see that $F^{(k)}_{m,\,n}=-F^{(k)}_{n+1}$ holds by
    the following manipulation of the determinants:
\begin{align*}
\|\ 0,\ \cdots,\ k-3,\ -1,\ n\ \| &= \|\ -1,\ 0,\ \cdots,\ k-3,\ n\ \| (-1)^{k-2} \\
&= \|\  0,\ \cdots,\ k-2, \ n+1\ \| (-1)^{k-2} \cdot 
\Big(\,{\textstyle\prod \phi_i}\,\Big)^{-1}\\[1ex]
&= \|\  0,\ \cdots,\ k-2, \ n+1\ \| (-1)^{k-2} (-1)^{k-1}.
\end{align*}

\medskip\noindent

(\ref{eq:7.1.2})\ \ For $0 \leq m \leq k-3$,   $F^{(k)}_{m,\,n}=0$ holds
    because the determinant $$\|\ 0,\ \cdots,\ k-3,\ m,\ n\ \|$$
    has duplicated columns, as $m \in \{ 0, \cdots, k-3\}.$

\medskip\noindent
(\ref{eq:7.1.3})\ \ For $m=k-2$, $F^{(k)}_{m,\,n}=F^{(k)}_{n}$ by 
    the generalized Binet formula \eqref{eq.binet} as 
\[
F^{(k)}_{m,\,n}= \frac{\|\ 0,\ \cdots,\ k-3,\ k-2,\ n\ \|}{\|\ 0,\ \cdots,\ k-2,\ k-1\ \|} = F^{(k)}_{n}
\]
    
\medskip\noindent
(\ref{eq:7.1.4})\ \ For $m=k-1$, $F^{(k)}_{m,\,n}=F^{(k)}_{n} - F^{(k)}_{n+1}$ by the 
    $k$-bonacci recursion 
    \[
    F^{(k)}_{k-1,\,n} = \sum_{i = k-2}^{-1} F^{(k)}_{i,\,n} =
    F^{(k)}_{k-2,\,n} + F^{(k)}_{-1,\,n} = F^{(k)}_{n} - F^{(k)}_{n+1}
    \,.
    \]

\medskip\noindent
(\ref{eq:7.1.5})\ \ For $m=k$, we apply the $k$-bonacci recursion to obtain
    \[
    F^{(k)}_{k,\, n} = \sum_{i=0}^{k-1} F^{(k)}_{i,\, n}
    = F^{(k)}_{k-2,\, n} + F^{(k)}_{k-1,\, n} =
    2 F^{(k)}_{n}  - F^{(k)}_{n+1}  \,.
    \]
    Since $F^{(k)}_n=2^{n-k} $ for $k \leq n \leq 2k-1$, 
    we have $F^{(k)}_{k,\,n}=0$ for $k \leq n \leq 2k-2$. 
Looking at the $k$-bonacci numbers, we see that the $k$-bonacci
    recursion gives us
    \[\begin{array}{ccccc}
    F^{(k)}_{-k} = 0, & \cdots, & F^{(k)}_{-3} = 0, &
    F^{(k)}_{-2} = -1, & F^{(k)}_{-1} = 1.
    \end{array}
    \]
    Putting these facts together, we see that  
    $F^{(k)}_{m,\,n}=-2^{n-k}F^{(k)}_{m-k},$ for 
    $-1\leq m \leq k-2 $ and $ k \leq n \leq 2k-2$. Then by 
    anti-symmetry we have
    $F^{(k)}_{m,\,n}=2^{m-k}F^{(k)}_{n-k}$ for $k \leq m \leq 2k-2$
    and $-1\leq n \leq k-2 $.  Since the last equation holds for $k$
    consecutive values of $n$, $k$-bonacci recursion shows that it holds
    for all $n$.

\medskip\noindent
(\ref{eq:7.1.6})
Since $F^{(k)}_{m,n}$  obeys the $k$-bonacci recursion over $m$ by 
\eqref{co:orthogonal}, we can extend Lemma~\ref{further.calc}  to $m$ to acquire 
\[
2F^{(k)}_{k-3,n} = F^{(k)}_{k-2,n}+F^{(k)}_{-3,n}
\]
holds.
Since $2F^{(k)}_{k-3,n}=0$ by \eqref{eq:7.1.2}, 
we have $F^{(k)}_{-3,n}=-F^{(k)}_{n}$ by \eqref{eq:7.1.3}.
\end{proof}

\begin{remark}\rm 
Our generalization of Vajda's identity is expressed by 
equation \eqref{gen.vajda} and that equation  is complicated. 
In some sense, that outcome was inevitable, since we insisted that the 
left-hand side of the equation mimic Vajda's identity
for $k=2$.
Nevertheless, by making a small change
of notation and thinking of the $F^{(k)}_{m,\,n}$ as the entries
in an array,
we feel we can make the equation easier to comprehend and remember.
Since $k$ is fixed, the  array is 2-dimensional and we suppose $m$ 
increases downward and $n$ increases to
the right, as usual. Setting $m=n+q$, the 
summation on the right-hand side of \eqref{gen.vajda} adds together the $p$ values
of the array that lie on a diagonal going down and to the left  starting
in row $m$ and ending in column $n$. That is,  with the
notation $m=n+q$ the right-hand side of \eqref{gen.vajda} equals
$$
\sum_{i=0}^{p-1} F^{(k)}_{m+i,\,n+p-1-i}
\,.
$$
The diagonal in question determines a rectangle in the array with two 
diagonally opposite
corners $(m,n)$ and $(m+p,n+p)$, where we have taken a little liberty 
with the location of the  bottom right corner, as in Figure~\ref{fig:visual_remark}. 
Now, using the four coordinates of those two 
corners we form a determinant:
\[
\begin{vmatrix}
    F^{(k)}_{n} & F^{(k)}_{n+p} \\[1ex]
    F^{(k)}_{m} & F^{(k)}_{m+p}
\end{vmatrix}
\,,
\]
and we observe that the determinant equals the left-hand side of \eqref{gen.vajda}. 
Thus we can restate \eqref{gen.vajda} as
\begin{equation}\label{gen.vajda.2}
  \begin{vmatrix}
    F^{(k)}_{n} & F^{(k)}_{n+p} \\[1ex]
    F^{(k)}_{m} & F^{(k)}_{m+p}
\end{vmatrix}
= \sum_{i=0}^{p-1} F^{(k)}_{m+i,\,n+p-1-i}
\,.
\end{equation}
\end{remark}
\begin{figure}[ht]
\centering
\begin{center}
\[
\begin{NiceArray}[columns-width=auto,hvlines,first-row,first-col]{||ccc||c}
                                  & n     &               &               & n+p\\
\Hline 
\Hline 
m &                                &               & F^{(k)}_{m,\,n+p-1}         &\\ 
&                                                        & \iddots       &               &\\ 
&                                   F^{(k)}_{m+p-1,\,n}               &               &               &\\ 
\Hline 
\Hline
m+p 
\end{NiceArray}
\]
\caption{Initial Visualization of \eqref{gen.vajda.2}}\label{fig:visual_remark}
\end{center}
\end{figure}

We can also work backwards from our diagonal as follows. If we have a single diagonal element $F^{(k)}_{m,n}$ such that $m=n-1$, then we can see that this is an application of Cassini's formula. 
\begin{equation}\label{gen.ex.cassini}
  \begin{vmatrix}
    F^{(k)}_{n} & F^{(k)}_{n+1} \\[1ex]
    F^{(k)}_{n-1} & F^{(k)}_{n}
\end{vmatrix}
= F^{(k)}_{n-1,\,n}
\,.
\end{equation}
If we have $r$ elements on a diagonal beginning at $(n-1,n)$ and increasing towards the upper right, we have an application of Catalan's formula.
\begin{equation}
\label{gen.ex.catalan_sum}
  \begin{vmatrix}
    F^{(k)}_{n} & F^{(k)}_{n+r} \\[1ex]
    F^{(k)}_{n-1} & F^{(k)}_{n+r-1}
\end{vmatrix}
= \sum_{i=0}^{r-1} F^{(k)}_{n-1+i,\,n+r-1-i} \,.
\end{equation}
or
\begin{equation}
\label{gen.ex.catalan}
  \begin{vmatrix}
    F^{(k)}_{n} & F^{(k)}_{n+r} \\[1ex]
    F^{(k)}_{n-1} & F^{(k)}_{n+r-1}
\end{vmatrix}
=  F^{(k)}_{n-1,\, n} + F^{(k)}_{n-2,\, n+1} +\cdots+
F^{(k)}_{n-r,\, n+r-1}\,.
\end{equation}
Ascending diagonals that do not meet either of these criteria, can only be classified as applications of Vajda.

We can take advantage of the fact that $F^{(k)}_{n,-3}=F^{(k)}_{k-2,n}=F^{(k)}_n$ by \eqref{eq:7.1.2} and \eqref{eq:7.1.6}, in order to see the results on the grid, as the $k$-bonacci numbers are represented on the appropriate row and column, as in Figure~\ref{fig:visual}.   Figure~\ref{fig:visual2}
illustrates the particular case $m=11$, $n=3$, $p=3$, when $k=3$.
\begin{figure}[ht]
\centering
\begin{center}
\[
\begin{NiceArray}[columns-width=auto,hvlines,first-row,first-col]{cc||ccc||c}
&-3 & \\
k-2&                   &               & F^{(k)}_{k-2,\,n}     &               &               & F^{(k)}_{k-2,\,n+p}\\ 
&                   &               &                       & \vdots        &               &\\
\Hline 
\Hline 
& F^{(k)}_{m,\,-3}   & \cdots        &                       &               & F^{(k)}_{m,\,n+p-1}         &\\ 
&                   &               &                       & \iddots       &               &\\ 
&                   &               & F^{(k)}_{m+p-1,\,n}              &               &               &\\ 
\Hline 
\Hline
& F^{(k)}_{m+p,\,-3} &               &                       &               &               &\\ 
\end{NiceArray}
\]
\caption{Final Visualization of \eqref{gen.vajda.2}}\label{fig:visual}
\end{center}
\end{figure}

\begin{figure}[ht]
\centering
\begin{center}
\[
\begin{NiceArray}[columns-width=auto,hvlines,first-row,first-col]{cc||ccc||c}
& \\
&                  &               & F^{(3)}_{1,\,3}=1     &               &               & F^{(3)}_{1,\,6}=7 \\ 
&                   &               &                       & \vdots        &               &\\
\Hline 
\Hline 
&F^{(3)}_{11,\,-3}=149   & \cdots        &                       &               & 53         &\\ 
&                   &               &                       & -88       &               &\\ 
&                   &               & -81               &               &               &\\ 
\Hline 
\Hline
&F^{(3)}_{14,\,-3}=927 &               &                       &               &               &\\ 
\end{NiceArray}
\]
\caption{$\begin{vmatrix} 1 & 7 \\ 149 & 927 \end{vmatrix}=1\cdot927-7\cdot149=-116=53-88-81$.}\label{fig:visual2}
\end{center}
\end{figure}

%
%
\section{Recovering the classical  identities}
To show that our generalization of Vajda's identity reduces
to the original when $k=2$, we need to show that
the right-hand side of \eqref{gen.vajda} simplifies
to $(-1)^{n-1}\,F_p\, F_q$ in that case.
There are two results needed to accomplish that goal.
The first  is the 
following theorem---which may be interesting in its own right.
\begin{theorem} \label{th.fib.form} For $p\geq 0$ and $q\in\mathbb{Z}$, it holds that
\begin{equation}\label{fib.form}
    F_p\,F_q = \sum_{i=0}^{p-1} (-1)^i\, F_{p-1-2i+q}
    \,.
\end{equation}
\end{theorem}
\begin{proof}
In case $p=0$, both sides of \eqref{fib.form} vanish,
and when $p=1$, both sides equal $F_q$. 

Arguing by induction, we consider $p+1$, for $p\geq1$.
We have
\begin{align*}
    F_{p+1}\, F_q &= F_p\, F_q + F_{p-1}\, F_q 
    = \sum_{i=0}^{p-1} (-1)^i\, F_{p-1-2i+q}
    + \sum_{i=0}^{p-2} (-1)^i\, F_{p-2-2i+q}\\[1ex]
    &= \sum_{i=0}^{p-2} (-1)^i\,(F_{p-1-2i+q}
    +F_{p-2-2i+q} ) + (-1)^{p-1}\, F_{-(p-1)+q}\\[1ex]
    &= \sum_{i=0}^{p-2} (-1)^i\,F_{p-2i+q}
    + (-1)^{p-1}\, F_{-(p-1)+q}
\end{align*}
and finally we observe that
\[
(-1)^{p-1}\, F_{-(p-1)+q} = (-1)^{p-1} F_{(p+1)-1-2(p-1)+q}
+ (-1)^p\, F_{(p+1)-1-2(p-1)+q}
\,,
\]
completing the proof.
\end{proof}

The second result we need is the following simple identity:
\begin{lemma}\label{eval.two}
For $a,b\in \mathbb{Z}$, it holds that
$F^{(2)}_{a,b}= (-1)^a\,F_{b-a}$.
\end{lemma}
\begin{proof}
We observe that when 
$\phi_1$ and $\phi_2$ are the roots of the
Fibonacci characteristic polynomial, the equation
\[
    \begin{vmatrix}
    \phi_1^{a} & \phi_1^b \\
    \phi_2^{a} & \phi_2^b
    \end{vmatrix}=(\phi_1\phi_2)^{a} \begin{vmatrix}
    \phi_1^{0} & \phi_1^{b-a} \\
    \phi_2^{0} & \phi_2^{b-a}
    \end{vmatrix}
\] 
implies $F^{(2)}_{a,b}= (-1)^a\,F_{b-a}$.
\end{proof}

\begin{corollary}
When $k=2$, equation \eqref{gen.vajda} reduces to Vajda's identity
\[F_{n}F_{n+p+q}-F_{n+p}F_{n+q} 
    = (-1)^{n-1}\, F_p\, F_q
    \,.
\]
Likewise,  \eqref{gen.vajda}  reduces to the 
identity of Cassini and to the identity of  Catalan
with the appropriate choices of $p$ and $q$.
\end{corollary}
\begin{proof}
For $p \geq 0$, using Lemma~\ref{eval.two} and Theorem~\ref{th.fib.form}, we see that 
when $k=2$,
\eqref{gen.vajda} becomes
\begin{align} \nonumber
F_{n}F_{n+p+q}-F_{n+p}F_{n+q} 
&=  
(-1)^{n+q}\sum_{i=0}^{p-1} 
(-1)^{i}
F_{p-1-q-2i} \nonumber\\[1ex] 
\nonumber
&=  (-1)^{n+q} \,F_{p} \, F_{-q} \\[1ex]
\label{almost.there}
&=  (-1)^{n-1}\ \,F_{p} \, F_{q}
\,,
\end{align}
where we have used the fact that $F_{-q} = (-1)^{q-1}\,F_q$.

When $p<0$, we define $n'=n+p+q, p'=-1$, and $q'=-q$. Then
\begin{align*} \nonumber
F_{n}F_{n+p+q}-F_{n+p}F_{n+q} 
&=F_{n'}F_{n'+p'+q'}-F_{n'+p'}F_{n'+q'} \\
&=(-1)^{n'-1} \,F_{p'} \, F_{q'} \\
&= \left[(-1)^{n-1}(-1)^{p}(-1)^{q}\right]\left[(-1)^{p-1}F_{p}\right]\left[(-1)^{q-1}F_{q}\right] \\
&= (-1)^{n-1}\ \,F_{p} \, F_{q}
\,,
\end{align*}
As mentioned in the introduction, Catalan's identity is a special case of Vajda's identity where $p=r$ and $q=-r$. Finally, Cassini's identity is the case $r=1$ of Catalan's identity.
\end{proof}

%
%

\section{The Tribonacci numbers}

Lemma~\ref{eval.two} allows us to express the quantities
$F^{(2)}_{a,b}$ in terms of the Fibonacci numbers.
The next lemma gives a similar way to express the quantities
$F^{(3)}_{a,b}$ in terms of the Tribonacci numbers.

\begin{theorem}\label{eval.three}
For $a,b\in \mathbb{Z}$, it holds that
\begin{equation} \label{tri.goal}
F^{(3)}_{a,b}= F^{(3)}_{a-b} \, F^{(3)}_{1-b}
- F^{(3)}_{a-b+1} \, F^{(3)}_{-b}
\,.
\end{equation}
\end{theorem}
\begin{proof}
    By \eqref{gen.vajda1}, we have
\[
F^{(3)}_{-b,a-b}  =   F^{(3)}_{a-b}F^{(3)}_{1-b} -F^{(3)}_{a-b+1}F^{(3)}_{-b} \,.
\]
It then suffices to show $F^{(3)}_{a,b}=F^{(3)}_{-b,a-b}$.
By letting the $\phi_i$  be the roots of the characteristic
polynomial for the Tribonacci recursion, we observe that 
$ \prod \phi_i=1 $
and thus we have
\begin{equation*}
    V^{(3)}\,F_{ -b,a-b}^{(3)} = \|\ 0,\ -b,\ a-b\ \| 
    =  \|\ 0,\ a,\ b\ \| \cdot 
    \Big(\,{\textstyle\prod \phi_i}\,\Big)^{-b}(-1)^2\\[1ex]
    =   V^{(3)}\,F_{ a,b}^{(3)} \,.
\end{equation*}
\end{proof}

\begin{corollary}
\begin{equation}\label{tri.step.1}
F_{ n-1, n}^{(3)} = F_{ 1-n}^{(3)} 
\,.
\end{equation}
\end{corollary}

\begin{proof}
    Since $F^{(3)}_{-1} = 1$ and $F^{(3)}_0=0$ hold, we 
see that \eqref{tri.goal}  implies \eqref{tri.step.1} when \mbox{$a=n-1,b=n$}.
\end{proof}

\begin{corollary}[Cassini's identity for the 
Tribonacci numbers]
It holds that
\begin{equation}
    \label{eq:3}
    \left(F^{(3)}_{n}\right)^{2}-F^{(3)}_{n-1}\,F^{(3)}_{n+1}=F^{(3)}_{1-n} \,.
\end{equation}
\end{corollary}
\begin{proof}
    This is an immediate consequence of equations \eqref{gen.cassini} and \eqref{tri.step.1}.
\end{proof}
\begin{remark}\rm
Equation \eqref{eq:3} 
confirms the conjecture in 
\cite{Shannon2024Simson} that the
right-hand side of the Tribonacci generalization of Cassini's 
(also called Simpson's) identity is $F^{(3)}_{1-n} $, which
is called  the 
reflected Tribonacci numbers in \cite[A057597]{oeis}.
\end{remark}

We know that the quantities $F^{(3)}_{a,\,b}$ satisfy the Tribonacci
recursion in each index separately. The next corollary shows us that
there is also a recursion that applies when both indices increase simultaneously.
\begin{corollary}
For $a,b\in \mathbb{Z}$, it holds that
\begin{equation}\label{trib.anti}
F^{(3)}_{a,\,b} = - F^{(3)}_{a-1,\,b-1} - F^{(3)}_{a-2,\,b-2}
+ F^{(3)}_{a-3,\,b-3}
\,.
\end{equation}
\end{corollary}

\begin{proof}
Taking advantage of the fact that $(a-i)-(b-i)= a-b$ for $i=1,2,3$,
we expand the right-hand side of \eqref{trib.anti} using \eqref{tri.goal}
and then use the Tribonacci recursion
to obtain
\begin{align*}
& - F^{(3)}_{a-1,\,b-1} - F^{(3)}_{a-2,\,b-2}
+ F^{(3)}_{a-3,\,b-3} \\[1ex]
&\qquad = -F^{(3)}_{a-b} \, F^{(3)}_{2-b}
+ F^{(3)}_{a-b+1} \, F^{(3)}_{1-b}
- F^{(3)}_{a-b} \, F^{(3)}_{3-b}+ F^{(3)}_{a-b+1}\, F^{(3)}_{2-b} 
+ F^{(3)}_{a-b} \,F^{(3)}_{4-b} - F^{(3)}_{a-b+1}\, F^{(3)}_{3-b}
\\[1ex]
&\qquad =  F^{(3)}_{a-b} \, \Big( F^{(3)}_{4-b}-F^{(3)}_{3-b}-F^{(3)}_{2-b} \Big)
- F^{(3)}_{a-b+1}\,\Big( -F^{(3)}_{3-b}+ F^{(3)}_{2-b}+ F^{(3)}_{1-b}\Big) \\[1ex]
&\qquad =  F^{(3)}_{a-b} \, F^{(3)}_{1-b} 
- F^{(3)}_{a-b+1}F^{(3)}_{-b} = F^{(3)}_{a,\,b}
\end{align*}
\end{proof}

%
%

\section{Further Results for Calculating}

The first two results in this section show how the value of 
$F^{(k)}_{m,\,n}$ can be expressed in terms of $k$-bonacci numbers. 

\begin{theorem}\label{th:vajda1}  For integers $m$  and $n$,
it holds that
\begin{equation}\label{gen.vajda1}
F^{(k)}_{m,n}  =   F^{(k)}_{n}F^{(k)}_{m+1} -F^{(k)}_{n+1}F^{(k)}_{m} \,.
\end{equation}
\end{theorem}
\begin{proof}
Apply \eqref{gen.vajda} with $p=1$ and $q= m-n$.
\end{proof}

\begin{theorem}\label{thm.reduce.to.k}
For integers $m$ and $n$, it holds that
\begin{equation} \label{eq:k}
F^{(k)}_{m,n}  
=   F^{(k)}_{n-k}F^{(k)}_{m} -F^{(k)}_{n}F^{(k)}_{m-k} \,.
\end{equation}
\end{theorem}
\begin{proof} We have 
\[
2F^{(k)}_{n} = F^{(k)}_{n+1}+F^{(k)}_{n-k}
\]
by \eqref{further.calc}.
Likewise, we have
\begin{equation}\label{two.tricks}
{\textstyle\frac12}\left(F^{(k)}_{m+1}+F^{(k)}_{m-k}\right) 
= F^{(k)}_{m}
\,.
\end{equation}
Multiplying \eqref{further.calc} and \eqref{two.tricks}, we obtain
\[
F^{(k)}_{n} \left( F^{(k)}_{m+1}+F^{(k)}_{m-k}\right) = F^{(k)}_{m} 
\left( F^{(k)}_{n+1}+F^{(k)}_{n-k}\right)
\,,
\]
which is equivalent to 
\[
F^{(k)}_{n}F^{(k)}_{m+1}-F^{(k)}_{n+1}F^{(k)}_{m} 
= F^{(k)}_{n-k}F^{(k)}_{m}
-F^{(k)}_{n}F^{(k)}_{m-k}
\,,
\]
and by Theorem~\ref{th:vajda1}, the left-hand side of this
last equation equals $ F^{(k)}_{m,n} $.
\end{proof}

The result in the following theorem is surprising. 
Essentially, it says that the sum of $k$ diagonal 
elements, as in
Figure~\ref{fig:visual3}, will appear beyond the lower 
right corner of the smallest square containing the diagonal.

\begin{theorem}
\label{lemma.surprise2}
For $k \geq 2$ and  $m,n\in \mathbb{Z}$,  it holds that
\begin{equation} 
F^{(k)}_{m,\,n} = \sum_{i=0}^{k-1}F^{(k)}_{m-k+i,\,n-1-i} 
\,. 
\end{equation}
\end{theorem}
\begin{proof}
We start by replacing $n$ by $n-k$ in  Vajda's identity for the $k$-bonacci
numbers, \eqref{gen.vajda}, to obtain
\[
F^{(k)}_{n-k}F^{(k)}_{n-k+p+q}-F^{(k)}_{n-k+p}F^{(k)}_{n-k+q} = 
\sum_{i=0}^{p-1}F^{(k)}_{n-k+q+i,\, n-k+p-1-i}
\,.
\]
Then we make the choices $p=k$ and $q=m-n$ to obtain
\[
F^{(k)}_{n-k}F^{(k)}_{m}-F^{(k)}_{n}F^{(k)}_{m-k} = 
\sum_{i=0}^{p-1}F^{(k)}_{m-k+i,\, n-1-i}
\,.
\]
By Theorem~\ref{thm.reduce.to.k}, the left-hand side of the last 
equation is equal to $F^{(k)}_{m,\,n}$, completing the proof.
\end{proof}

\begin{figure}[ht]
\centering
\begin{center}
\[
\begin{NiceArray}[columns-width=auto,hvlines,first-row,first-col]{cc||ccc||c}
 &-3 & & n-k &&& n\\
k-2&                   &               & F^{(k)}_{k-2,n-k}     &               &               & F^{(k)}_{k-2,n}\\ 
&                   &               &                       & \vdots        &               &\\
\Hline 
\Hline 
m-k& F^{(k)}_{m-k,-3}   & \cdots        &                       &               & a_{0}         &\\ 
&                   &               &                       & \iddots       &               &\\ 
&                   &               & a_{k-1}               &               &               &\\ 
\Hline 
\Hline
m& F^{(k)}_{m,-3} &               &                       &               &               & \sum_{i=0}^{k-1}a_i\\ 
\end{NiceArray}
\]
\[
a_i=F^{(k)}_{m-k+i,\, n-1-i}
\]
\caption{$k$-diagonal visualization}\label{fig:visual3}
\end{center}
\end{figure}

%
%
\section{Applications}
We note that   our method and results also apply  more generally
to a constant coefficient homogeneous linear recursion, $G_n$,
provided that the characteristic polynomial of the recursion has simple roots
and that the initial values are
\[
G_n = \begin{cases}
    0 & \text{if\ \ } 0\leq n < k-1, \\
    1  & \text{if\ \ } n = k-1.
\end{cases}
\]
Following \eqref{eq:Binet}, for $\ell\leq k$, we can use the  roots $\phi_1, \phi_2,\dots,\phi_k$ to define
\[
G_{c_1,\ \dots \ c_\ell} = \left. 
\frac{V^{(k)}_{c_1,\ \dots \ c_\ell}}{V^{(k)}} \right|_{\phi_1,\dots,\phi_k}
\]

A large class of such recursions is provided by the Lucas Sequences 
of the first kind $U_{n}(P,Q)$, given by
\[
U_{n}(P,Q) = \begin{cases}
    0 & \text{if\ \ } n=0, \\
    1  & \text{if\ \ } n = 1, \\
    P\,U_{n-1}(P,Q)-Q\, U_{n-2}(P,Q)& \text{if\ \ } n > 1,
\end{cases}
\]
where $P$ and $Q$ are integers. 
The characteristic polynomial of the recursion is 
$x^2-Px + Q$, so
\begin{equation}\label{non0disc}
P^2- 4Q\neq 0
\end{equation}
guarantees that the roots are simple.
Since  $k=2$, we have only 
$U_{n}(P,Q)$ and $U_{a,b}(P,Q)$ to consider.

Applying Theorem~\ref{lemma:second_} as we did before,
with $p=1$ and $q=-1$,
we obtain
\[
U_{n}(P,Q)^{2}-U_{n-1}(P,Q)\, U_{n+1}(P,Q) = U_{n-1,n}(P,Q)
\,.
\]
Also, we have
\[
    V^{(2)}\,U_{n-1,n}(P,Q) = \|\ n-1,\ n\ \| 
    =  \|\ 0,\ 1\ \| \cdot 
    \Big(\,{\textstyle\prod \phi_i}\,\Big)^{n-1}
= V^{(2)}\, Q^{n-1} \,,
\]
since the product of the roots is $Q$. Thus we have
the following result:
\begin{proposition}[Cassini's Formula 
for Lucas Sequences of the First Kind] 
If \eqref{non0disc} holds, then
\begin{equation}\label{eq:lucasfirst}
    U_{n}(P,Q)^2-U_{n-1}(P,Q)\ U_{n+1}(P,Q)=Q^{n-1}
    \,.
\end{equation}
\end{proposition}
We have seen this fact mentioned in the 
literature for sequences having various common names 
that are examples of  
Lucas Sequences of the first kind. Indeed,  the Fibonacci numbers
themselves satisfy $F_n= U_n(1,-1)$, so \eqref{eq:lucasfirst}
is equivalent to  \eqref{eq:cassini}. The following are some
examples:
\begin{enumerate}
    \item In \cite{koken2008jacobsthal}, it was shown that the 
    Jacobsthal numbers, $J_n= U_n(1, -2)$, satisfy 
\[
J_n^2-J_{n-1}J_{n+1}=(-2)^{n-1}.
\] 
    \item In \cite{horadam1971pell}, it was shown that the 
    Pell numbers, $P_n= U_n(2, -1)$, satisfy 
\[
P_n^2-P_{n-1}P_{n+1}=(-1)^{n-1}.
\] 
    \item In \cite{ochalik2018generalized}, it was shown that the 
    Mersenne numbers, $M_n= U_n(3, 2)$, satisfy 
\[
M_n^2-M_{n-1}M_{n+1}=2^{n-1}.
\] 
    \item In \cite{ray2014generalization}\footnote{in 
the context of \it{Balancing Numbers}, which are $U_n(6\lambda, 1)$, 
for $\lambda \geq 1$},  it was shown that the square roots of
the square 
triangular numbers \cite[A001109]{oeis}, $S_n= U_n(6, 1)$, satisfy 
\[
S_n^2-S_{n-1}S_{n+1}=1.
\] 
\end{enumerate}

\begin{remark}\rm
In \cite{jeffery2014divisibility} (in the paragraph before
Lemma~14), the authors obtain a Cassini result  for the 
Lucas Sequences of the second kind. For Lucas Sequences of the first 
kind, $U_n = U_n(P,Q)$, they give the  matrix identity 
\begin{equation}\label{eq:matrix}
\begin{pmatrix}
    U_{n+1} & -QU_n \\
    U_n & -QU_{n-1} 
\end{pmatrix}=
\begin{pmatrix}
P & -Q \\
1 & 0
\end{pmatrix}^{n}
\,,
\end{equation}
citing \cite{brenner1951lucas} (though the notation 
in \cite{brenner1951lucas} is different). 
Taking the determinant of both sides of
\eqref{eq:matrix}, one obtains \eqref{eq:lucasfirst}---though this 
was left unstated in both 
references.

Taking $P=1$ and $Q=-1$, leads us to 
what is known as the 
Fibonacci  $\mathcal{Q}$-matrix (see \cite{Gould01081981}).
This  gives  to another 
way to generalize Cassini's identity.
The $\mathcal{Q}$-matrix is given by  
\[
\mathcal{Q}=
\begin{pmatrix}
    1 & 1 \\
    1 & 0
\end{pmatrix}
=
\begin{pmatrix}
    F_2 & F_1 \\
    F_1 & F_0
\end{pmatrix}.
\]
By induction, we see that 
\begin{equation}\label{eq:q2}
\begin{pmatrix}
    F_{n+1} \\
    F_n
\end{pmatrix} = 
   \mathcal{Q}^n 
\begin{pmatrix}
    F_1 \\
    F_0
\end{pmatrix}
=
   \mathcal{Q}^n 
\begin{pmatrix}
    1 \\
    0
\end{pmatrix}
\hbox{\rm\quad  and\quad  }
    \mathcal{Q}^n=
\begin{pmatrix}
    F_{n+1} & F_{n} \\
    F_{n} & F_{n-1}
\end{pmatrix}.
\end{equation}
The eigenvalues of $\mathcal{Q}$ are the roots of the Fibonacci 
characteristic 
polynomial, which leads to an alternative proof of the Binet formula. 
In addition, taking determinants and noting that  
$\hbox{\rm det\,}[\mathcal{Q}]=-1$, we obtain
Cassini's identity. 

To extend to the $k$-bonacci numbers, we begin by defining
the $k\times k$ matrix $\mathcal{Q}^{(k)}$ by
\begin{equation}
\mathcal{Q}^{(k)}=
\begin{pmatrix}
    1 & 1 & \cdots & 1 & 1\\
    1 & 0 & \cdots & 0 & 0\\
    0 & 1 & \cdots & 0 & 0\\
    \vdots &&\ddots& &\vdots \\
    0 & &\cdots & 1 & 0\\
\end{pmatrix}
\end{equation}
By induction, we see that
\[
\begin{pmatrix}
F^{(k)}_{n+k-1} \\
F^{(k)}_{n+k-2} \\
\vdots \\
F^{(k)}_n
\end{pmatrix}
=\left(\mathcal{Q}^{(k)}\right)^n
\begin{pmatrix}
F^{(k)}_{k-1} \\
F^{(k)}_{k-2} \\
\vdots \\
F^{(k)}_0
\end{pmatrix}
=
\left(\mathcal{Q}^{(k)}\right)^n
\begin{pmatrix}
1 \\
0 \\
\vdots \\
0
\end{pmatrix}.
\]
By Laplace (i.e., co-factor) expansion along the last column,
we obtain
\begin{equation} \label{eq:det}
    \hbox{\rm det\,}\left[\mathcal{Q}^{(k)}\right]
    =(-1)^{k-1}  \hbox{\rm det\,}[ I ]=(-1)^{k-1}.    
\end{equation}

\noindent We can also calculate the eigenvalues of $\mathcal{Q}^{(k)}$. To do so, we will show that
$$\hbox{\rm det\,}\left[\mathcal{Q}^{(k)}-\lambda I\right]=(-1)^k(\lambda^k-\lambda^{k-1}-\cdots-1),$$ 
which suffices to show that the eigenvalues of $\mathcal{Q}^{(k)}$ are the roots 
of the $k$-bonacci characteristic polynomial. We first note by an 
elementary calculation that
\[
\hbox{\rm det\,}\left[\mathcal{Q}-\lambda I\right]
= \lambda^2-\lambda-1=(-1)^2(\lambda^2-\lambda-1).
\]
We can then calculate $\hbox{\rm det\,}\left[\mathcal{Q}^{(k)}-
\lambda I\right]$ by expansion along the 
last column, and find that
\begin{align*}
\hbox{\rm det\,}\left[\mathcal{Q}^{(k)}-\lambda I\right] 
&= (-1)^{k-1}+(-1)^{2(k-1)}(-\lambda)\, 
\hbox{\rm det\,}\left[\mathcal{Q}^{(k-1)}-\lambda I\right] \\
&= (-1)^{k} \left( \lambda (\lambda^{k-1}-\lambda^{k-2}\cdots-1)   
-1\right) \\
&= (-1)^{k}\left(\lambda^k-\lambda^{k-1}-\cdots-1\right).
\end{align*}

For all $k$, an identity similar to \eqref{eq:q2} can be achieved by 
linear algebra. 
In \cite{cerda2018three}, this idea is realized for $k=3$, as
\[
\left(\mathcal{Q}^{(3)}\right)^{n-1}=\begin{pmatrix}    
F^{(3)}_{n+1} & F^{(3)}_{n} + F^{(3)}_{n-1} & F^{(3)}_{n} \\
F^{(3)}_{n} & F^{(3)}_{n-1}+F^{(3)}_{n-2} & F^{(3)}_{n-1} \\
F^{(3)}_{n-1} & F^{(3)}_{n-2} + F^{(3)}_{n-3} & F^{(3)}_{n-2}
\end{pmatrix},
\] 
leading to the following Cassini-style identity 
for the Tribonacci numbers
\begin{equation}
F^{(3)}_{n-1} + (F^{(3)}_{n-2})^2F^{(3)}_{n+1}+F^{(3)}_{n-3}
(F^{(3)}_{n})^2-2F^{(3)}_{n-2}F^{(3)}_{n-1}F^{(3)}_{n}-
F^{(3)}_{n-3}F^{(3)}_{n-1}F^{(3)}_{n+1}=1.
\end{equation}
\end{remark}


\bibliographystyle{plain} 
\bibliography{kbonacci} 

\end{document}